\documentclass[11pt]{amsart}

\usepackage[T1]{fontenc}

\usepackage[utf8]{inputenc}
\usepackage[english,russian]{babel}

\usepackage{mathrsfs}
\usepackage{amsmath}
\usepackage{amssymb}
\usepackage{amsthm}
\usepackage{tikz-cd}
\usepackage{systeme}
\usepackage[margin=1in]{geometry}
\usepackage{url}
\usepackage{mathtools}
\usepackage{bm}
\usepackage{calligra}
\usepackage{multicol}
\usepackage{subfig}
\usepackage{float}
\usepackage[colorlinks = true,
linkcolor = blue,
urlcolor  = purple,
citecolor = blue,
anchorcolor = blue]{hyperref}

\graphicspath{ {pictures/} }

\theoremstyle{plain}
\newtheorem{theorem}{Theorem}[section]

\newtheorem{proposition}[theorem]{Proposition}
\newtheorem{lemma}[theorem]{Lemma}

\theoremstyle{definition}
\newtheorem{definition}[theorem]{Definition}
\newtheorem{remark}[theorem]{Remark}

\newtheorem{example}[theorem]{Example}

\numberwithin{equation}{theorem}

\newcommand{\ZZ}{\boldsymbol{\mathbb{Z}}}

\newcommand{\cc}[1]{\mathcal{#1}}

\newcommand{\bb}[1]{\mathbb{#1}}
\newcommand{\OO}{\mathcal{O}}

\newcommand\dual{\raise0.2ex\hbox{$\scriptscriptstyle\vee$}}

\newcommand*{\sheafhom}{\mathscr{H}\text{\kern -3pt {\calligra\large om}}\,}

\begin{document}

\selectlanguage{english}

\sloppy

\date{September 2020}
\author[D. Pedchenko]{Dmitrii Pedchenko}
\address{Department of Mathematics, The Pennsylvania State University, University Park, PA 16802}
\email{dzp5326@psu.edu}


\title{The Arens-Michael envelopes of the Jordanian Plane and $U_q(\mathfrak{sl}(2))$}

\begin{abstract}
The Arens-Michael functor in noncommutative geometry is an  analogue of the analytification functor in algebraic geometry: out of the ring of ``algebraic functions'' on a noncommutative space it constructs the ring of ``holomorphic functions'' on it. In this paper, we explicitly compute the Arens-Michael envelopes of the Jordanian plane and the quantum enveloping algebra $U_q(\mathfrak{sl}(2))$ of $\mathfrak{sl}(2)$ for $|q|=1$.
\\

\noindent This is an article version of the author's senior thesis \cite{Ped15} at HSE University from 2015.  \end{abstract}

\maketitle

\tableofcontents

\section{Introduction}

Noncommutative geometry appeared in the second half of the XX century as the  result of rethinking and conceptualizing classic facts about the interplay between geometry and algebra. Oftentimes all essential information about a geometric space (e.g., affine algebraic variety, smooth manifold, topological space) is  contained in a suitably chosen algebra of functions (polynomial, smooth, continuous) on the space. This observation is formalized in a theorem which establishes the anti-equivalence between the category of spaces at hand and the respective category of algebras of functions on such spaces. Sometimes the resulting category of commutative algebras admits an abstract description which a priori does not bind elements of the algebra to functions on any space. To illustrate this point, consider Hilbert's Nullstellensatz which implies that for an algebraically closed field the category of affine algebraic varieties is anti-equivalent to the category of commutative finitely generated reduced algebras. The Gelfand-Naimark Theorem implies that  the category of compact Hausdorff topological spaces is anti-equivalent to the category of commutative $C^*$-algebras.

The basic idea of noncommutative geometry is to view an arbitrary (noncommutative) algebra as an ``algebra of functions on a noncommutave space''. This idea is based on an observation that many important geometric concepts and constructions stated in algebraic terms remain meaningful for noncommutative algebras providing us with the tools and intuition for studying these algebras.

``Noncommutative geometry'' includes several mathematical disciplines which have different research objects but are unified by that aforementioned idea. Noncommutative measure theory studies von Neumann algebras (note that commutative von Neumann algebras are exactly the algebras of essentially bounded measurable functions on measure spaces); noncommutative topology studies  $C^*$-algebras (because of the Gelfand-Naimark theorem mentioned above); noncommutative affine algebraic geometry studies finitely generated algebras; noncommutative differential geometry studies dense subalgebras of  $C^*$-algebras equipped with a special ``differential'' structure.

Notice that an important discipline is missing from this list - noncommutative complex analytic geometry which in the commutative case bridges differential and algebraic geometries. One of the main reasons of why this theory is underdeveloped is that it is unclear which algebras should be considered the noncommutative generalizations of the algebras of holomorphic functions on complex analytic spaces. The ideal starting point for the development of any kind of noncommutate geometry is having a category $\mathscr{A}$ of associative algebras such that the full subcategory of $\mathscr{A}$ consisting of commutative algebras is anti-equivalent to a certain category $\mathscr{C}$ of ``spaces''. In such a situation we may think of algebras belonging to $\mathscr{A}$ as the noncommutative analogues of the spaces belonging to $\mathscr{C}$. Therefore, in the case of noncommutative complex analytc geometry we would like to start with some category consisting of  algebras such that the commutative ones are exactly the algebras of holomorphic functions on complex analytic spaces. As Pirkovskii states in \cite{Pir08}, apparently in full generality such a class of algebras  has  not yet been introduced.

However, we can simplify our problem as follows. Any affine algebraic variety over $\mathbb{C}$ is certainly a complex analytic space. So we can narrow down the category $\mathscr{C}$ of complex analytic spaces to the category of affine algebraic varieties over $\mathbb{C}$ viewed as complex analytic spaces.  It turns out that in this case we have a construction that for each finitely generated $\mathbb{C}$-algebra $A$ (deemed as the ``algebra of regular algebraic functions on a noncommutative affine scheme of finite type'') assigns a new algebra  $\widehat {A}$ (deemed as the ``algebra of holomorphic functions on that scheme'') such that if our algebra was a commutative algebra of regular algebraic functions $A = \mathcal {O}^{alg}(X)$  on an affine scheme $X$ of finite type, then we get the algebra of holomorphic functions on  $X$: $\widehat {A}=\mathcal {O}^{hol}(X)$.

The resulting algebra  $\widehat {A}$ is known as  the Arens-Michael envelope of algebra $A$.  Therefore, we view the Arens-Michael envelopes of finitely generated algebras as the algebras of holomorphic functions on noncommutative affine schemes of finite type. So far, the Arens-Michael envelopes are explicitly known only for  a handful of noncommutative algebras (see \S \ref{examples} and \cite[\S 5]{Pir08b}), and it is an important for the development of the theory and its scope of applicability to grow the body of examples. 

In this paper, we add two algebras to the list of examples: we explicitly compute the Arens-Michael envelopes of the Jordanian plane and the quantum enveloping algebra $U_q(\mathfrak{sl}(2))$ for $|q|=1$ following the recently developed techniques in this area.

Finally, let us mention that this paper is an improved version of the author's senior thesis \cite{Ped15} at the National Research University - Higher School of Economics from 2015.

\subsection*{Organization of the paper} In \S \ref{AM functor}, we review the basic definitions pertaining to the Arens-Michael functor. In \S \ref{examples}, we recall some known examples of algebras for which the Arens-Michal envelope is explicitly known. In \S \ref{theory}, we review the theoretical constructions necessary for our computations following \cite{Pir08b}.

Finally, we present our key results in \S \ref{JP} and \S \ref{EA}. In \S \ref{JP}, we explicitly compute the Arens-Michael envelope of the Jordanian plane. In \S \ref{EA}, we explicitly compute the Arens-Michael envelope of
the quantum enveloping algebra $U_q(\mathfrak{sl}(2))$ for $|q|=1$.

\subsection*{Acknowledgments} We would like to thank Alexei Pirkovskii for inspiring academic advising and helpful discussions.
\section{The Arens-Michael functor}\label{AM functor}

In what follows all vector spaces and algebras are taken over the field of complex numbers $\mathbb{C}$; all algebras are assumed to be associative and unital. The seminorm $\| \cdot \|$ on an algebra $A$ is called {\it submultiplicative} if $\| ab \|\leq\|a\|\|b\|$ for all $a,b\in A$. A complete topological algebra with a topology generated by a family of submultiplicative seminorms is called an {\it Arens-Michael algebra}. We start with our main definitions.

\begin{definition}\label{AM}  Let $A$ be a topological  algebra. The pair $(\widehat{A},\iota_A)$ , consisting of an Arens-Michael algebra  $\widehat {A}$ and a continuous homomorphism $\iota_A: A \to \widehat {A}$, is called {\it the Arens-Michael envelope} of algebra $A$ if for an arbitrary Arens-Michael algebra {\it B} and an arbitrary continuous homomorphism $\varphi : A \to B$, there exists a unique continuous homomorphism $\widehat {\varphi}: \widehat {A} \to B$ such that the following diagram commutes:

\[
\begin{tikzcd}
\widehat {A} \arrow[dashed] {r}{\widehat {\varphi}} & B \\
A \arrow[swap]{ur}{\varphi} \arrow{u}{\iota_A}
\end{tikzcd}
\]
\end{definition}
Arens-Michael envelopes were introduced by Taylor \cite{Tay72}, but here we are using the terminology of Helemskii \cite{Hel89} and Pirkovskii \cite[\S 3]{Pir08b}.

It is clear from the definition that the Arens-Michael envelope is unique up to a unique isomorphism of topological algebras over $A$. Moreover, it always exists \cite{Tay72}, and it can be obtained as the completion of $A$ with respect to all continuous submultiplicative seminorms on $A$. Note that the topology induced by submultiplicative seminorms might be non-Hausdorff so that before taking the completion we should take the quotient by the closure of $\{0\}$. Therefore, the canonical homomorphism  $\iota_A: A \to \widehat {A}$ might have a nontrivial kernel. 

\begin{definition}\label{AM2}
Let $A$ be an algebra without a topology. The \emph{Arens-Michael envelope} of $A$ is the Arens-Michael envelope of $A$ endowed with the strongest locally convex topology in the sense of Definition \ref{AM}.
\end{definition}

Finally, let us note that the association $A \mapsto \widehat {A}$ extends to algebra homomorphisms $A \to B$ so that we obtain a functor from the category of algebras to the category of Arens-Michael algebras {\it (the Arens-Michael functor)}.

\section{Examples of the Arens-Michael envelopes}\label{examples}

Next we discuss some known examples of the Arens-Michael envelopes. 

The next example and proposition justify our assertions from the introduction about the Arens-Michael functor being a noncommutative analogue of the analytification functor in algebraic geometry.

\begin{example} As was noted by Taylor \cite{Tay72}, the Arens-Michael envelope of the polynomial algebra $\mathbb{C}[x_1,...,x_n]={\mathcal {O}}^{alg}(\mathbb{C}^n)$ is the algebra of holomorphic functions $\mathcal {O}^{hol}(\mathbb{C}^n)$ with compact-open topology. 
\end{example}

Pirkovskii  generalized this to statement to affine algebraic varieties.
\begin{proposition}[{\cite[Example 3.6]{Pir08b}}] Let $X$ be an affine algebraic variety over $\mathbb{C}$ and let ${A=\mathcal {O}^{alg}(X)}$ be the algebra of regular algebraic functions on $X$. The Arens-Michael envelope of $A$ is the algebra $\mathcal {O}^{hol}(X_{an})$ of holomophic functions on $X$ when we view $X$ as a complex analytic space, with compact-open topology. The same is true for the affine schemes of finite type over $\mathbb {C}$. 
\end{proposition}

From this proposition we see that  the geometric analytification functor associating to an affine algebraic scheme $X$ a complex analytic space $X_{an}$ corresponds to  the algebraic or functional-analytic Arens-Michael functor when we instead work with functions on those spaces. As we noted in the introduction, the finitely generated noncommutative algebras are the natural candidates for the noncommutative affine schemes of finite type, so  we view the Arens-Michael envelopes of the finitely generated noncommutative algebras as the algebras of holomorphic functions on the noncommutative affine schemes of finite type.

Here is the ``most noncommutative'' example:

\begin{example}[The free algebra] Let $F_n=\mathbb{C}\langle x_1,...,x_n \rangle$ be a free algebra with $n$ generators. For each $k$-tuple $\alpha=(\alpha_1,...,\alpha_k)$ of integers, $1 \leq \alpha_i \leq n$, set $x_\alpha=x_{\alpha_1} \cdot \cdot \cdot  x_{\alpha_k}$  and $| \alpha|=k$. Then each element of $F_n$ is written as a noncommutative polynomial $\sum\nolimits_{|\alpha| \leq N} c_\alpha x_\alpha$. Denote the set of all $\alpha$ as $W_n$.

Taylor \cite{Tay72} showed that 
$$\widehat{F_n}=\{a=\sum \limits_{\alpha \in W_n} c_\alpha x_\alpha \quad : \quad \| a \|_\rho = \sum \limits_{\alpha \in W_n} | c_\alpha| \rho_\alpha < \infty  \quad  \text{for any} \quad \rho>0 \}.$$The topology on $\widehat{F_n}$ is defined by the family of seminorms $\{ \| \cdot \|_\rho  : \rho \in \bb{R}_{>0} \}$. 
\end{example}

\begin{example}[The quantum plane]\label{ex 1} Fix a complex number $q \in \mathbb {C} \setminus \{0\}$. {\it The quantum plane} is an algebra (denoted ${\mathcal {O}}^{alg}_q(\mathbb{C}^2)$)  with two generators $x,y$, subject to a relation $xy=qyx$.  The monomials $x^i y^j$  $(i,j \geq 0)$ form a basis of ${\mathcal {O}}^{alg}_q(\mathbb{C}^2)$ so that this algebra can be viewed as an algebra of polynomials with a ``twisted'' multiplication. 

\thickspace

Denote the Arens-Michael envelope of ${\mathcal {O}}^{alg}_q(\mathbb{C}^2)$ by ${\mathcal {O}}^{hol}_q(\mathbb{C}^2)$, and view it as an algebra of holomorphic functions on the quantum plane. The next result is due to Pirkovskii.
\begin{proposition}[{\cite[Corollary 5.14]{Pir08b}}]

Let $q \in \mathbb{C} \setminus \{0 \}$.

\begin{enumerate}
\item
 If $|q| \geq 1$,  then
$${\mathcal {O}}^{hol}_q(\mathbb{C}^2)=\{a=\sum \limits_{i,j=0}^{\infty} c_{ij} x^i y^j \quad : \quad  \| a \|_\rho = \sum \limits_{i,j=0}^{\infty} |c_{ij}| \rho^{i+j} < \infty\quad  \text{for any} \quad \rho>0 \}.$$

\item If $|q| \leq 1$, \it then

$${\mathcal {O}}^{hol}_q(\mathbb{C}^2)=\{a=\sum \limits_{i,j=0}^{\infty} c_{ij} x^i y^j \quad : \quad \| a \|_\rho = \sum \limits_{i,j=0}^{\infty} |c_{ij}| |q|^{ij} \rho^{i+j} < \infty\quad  \text{for any} \quad \rho>0 \}.$$

\end{enumerate}

In both cases the topology on ${\mathcal {O}}^{hol}_q(\mathbb{C}^2)$  is generated by the family of seminorms $\{\| \cdot \|_\rho : \rho \in \bb{R}_{>0} \}$  and the multiplication is defined by the relation $xy=qyx$.
\end{proposition}
\end{example}

\begin{example}\label{ex 2} Consider the universal enveloping algebra of the Lie algebra $\mathfrak{g}$ with basis $\{x,y\}$ and the commuting relation $[x,y]=y$. Due to the Poincare - Birkhoff - Witt theorem, the universal enveloping algebra $U(\mathfrak{g})$ can be viewed as a polynomial algebra with a ``twisted'' multiplication.  As shown in \cite[Example 5.1]{Pir08b}
$$\widehat{U}(\mathfrak{g}) = \{a=\sum \limits_{i,j=0}^{\infty} c_{ij} x^i y^j \quad : \quad  \sum \limits_{i=0}^{\infty} |c_{ij}| \rho^{i} < \infty, \quad  \forall j \in \mathbb{Z}_+, \quad \forall \rho >0 \}.$$
The topology of $\widehat{U}(\mathfrak{g})$ is generated by the family of seminorms 

$$\{ \| \cdot \|_{n, \rho} \quad : \quad \|\sum \limits_{i,j=0}^{\infty} c_{ij} x^i y^j \|_{n,\rho} = \sum \limits_{j=0}^{n} \sum \limits_{i=0}^{\infty} |c_{ij}| \rho^{i} < \infty,   \quad n \in \mathbb{Z}_+, \quad \rho \in \bb{R}_{>0} \}.$$
\end{example}

In the previous examples, the Arens-Michael envelopes of polynomial algebras with a ``twisted'' multiplications happened to be algebras of ``noncommutative power series'', as one would expect by comparing to the commutative case. Interestingly, the next example (due to Taylor \cite{Tay72}) shows that this is not always the case.

\begin{example}[The universal enveloping algebra of a semisimple Lie algebra]\label{lie matrix} Suppose $\mathfrak{g}$ is a semisimple Lie algebra. Every finite-dimensional irreducible representation $\pi_\lambda$ of algebra $\mathfrak{g}$ extends to a homomorphism  $$\pi_\lambda : U(\mathfrak{g}) \to M_{d_\lambda}(\mathbb{C}) \quad ({d_\lambda} = dim \ \pi_\lambda) .$$ If we denote the set of equivalence classes of irreducible finite-dimensional representations of $\mathfrak{g}$ by  $\widehat {\mathfrak{g}}$, we get a homomorphism 
$$\qquad \prod \limits_{\lambda \in \widehat {\mathfrak{g}}} \pi_\lambda : U(\mathfrak{g}) \to \prod \limits_{\lambda \in \widehat {\mathfrak{g}}} {M_{d_\lambda}}(\mathbb{C}).$$ 
The algebra   $\prod \limits_{\lambda \in \widehat {\mathfrak{g}}} {M_{d_\lambda}}(\mathbb{C})$ with the product topology and the homomorphism $\prod \limits_{\lambda \in \widehat {\mathfrak{g}}} \pi_\lambda$ form the Arens-Michael envelope of $U(\mathfrak{g})$.

This example is a bit discouraging since contrary to the above examples, this time the Arens-Michael envelope looks completely different from the initial algebra (for example, $U(\mathfrak{g})$ is an integral domain but  $\prod \limits_{\lambda \in \widehat {\mathfrak{g}}} {M_{d_\lambda}}(\mathbb{C})$ is not). Nonetheless, the canonical homomorphism $A \to \widehat{A}$ is injective (as it was in all other examples so far). 
\end{example}

The next example shows the worst possible situation.
\begin{example}[Weyl algebra] Weyl algebra $A$ is an algebra with two generators $x, \partial$ with the commuting relation $[\partial,x]=1$. It is well-known that in a non-zero normed algebra there are no elements with this commuting relation. Therefore $\widehat{A} =0$ and the canonical homomorphism is not injective.

It is interesting to note that if we quantize Weyl algebra by taking the commuting relation to be $$\partial x - q x \partial =1 \quad (q \neq 0,1),$$ the resulting Arens-Michael envelope would again be  the algebra of ``noncommutative'' polynomials (see \cite[Corollary 5.19]{Pir08b}). 

\end{example}

For more examples, see \cite{Pir08, Pir08b}.

\section{Theoretical constructions}\label{theory} In this section, we collect the  theoretical facts necessary for our computations following \cite{Pir04} and \cite{Pir08b}. We will be referring to a complete, Hausdorff, locally convex topological algebra with jointly continuous multiplication as  a $\widehat {\otimes}$-\emph{algebra}.

\subsection{The Arens-Michael envelopes and tensor product}\label{tensor}
First, we recall how to describe the topology on the projective tensor product of two $\widehat {\otimes}$-modules.

\begin{proposition}[{\cite[Proposition 2.3 (vi]{Pir08b}}]\label{tensor 1} Suppose $A$ is a $\widehat {\otimes}$-algebra, $X$ is a right $A$-$\widehat {\otimes}$-module, $Y$ is left $A$-$\widehat {\otimes}$-module. Furthermore, suppose that both $X$ and $Y$ have countable or finite dimension and the topology on $X$ and $Y$ is the strongest locally convex topology. Then the algebraic tensor product $X \otimes_A Y$ with the strongest locally convex topology coincides with the projective tensor product $X \widehat {\otimes}_A Y$.  

\end{proposition}

The next proposition shows that the Arens-Michael envelope of the projective tensor product of two $\widehat {\otimes}$-algebras can be computed as the projective tensor product of  their Arens-Michael envelopes.

\begin{proposition}[{\cite[Proposition 6.4]{Pir04}}]\label{tensor 2} Let $A,B$ be $\widehat {\otimes}$-algebras. Then there exists a topological algebra isomorphism $$(A \widehat{\otimes} B)^{\widehat{ \ }} \cong \widehat{A} \widehat{\otimes} \widehat{B}.$$
\end{proposition}
In other words, the operations of taking the Arens-Michael envelope and taking the projective tensor product can be interchanged.

\subsection{The Arens-Michael envelopes and Ore extensions}\label{Ore}
The computation of the Arens-Michael envelopes of many polynomial algebras (including Examples \ref{ex 1} and \ref{ex 2}) is greatly facilitated by a theoretical construction known as the {\it Ore extension}.

\subsubsection{Algebraic Ore extensions} First, we consider a purely algebraic construction.

\begin{definition}\label{twisted mult} Let $R$ be an associative $\bb{C}$-algebra (without a topology) and $\alpha: R \to R$ an algebra endomorphism. A $\bb{C}$-linear map $\delta :R \to R$ is called $\alpha$-{\it differentiation} if $$\delta (ab)=\delta(a)b+\alpha(a)\delta(b)$$ for any $a,b \in R$. The {\it Ore extension} $R[z;\alpha, \delta]$ is a noncommutative algebra obtained by endowing the left $R$-module of polynomials $\sum \limits_{i=0}^{n} r_i z^i$ with a ``twisted''  multiplication with a relation \begin{equation*} zr=\alpha(r)z+\delta(r)\end{equation*} for $r \in R$. Note that the natural inclusions  $R \hookrightarrow R[z;\alpha, \delta]$ and $\mathbb{C}[z] \hookrightarrow R[z;\alpha, \delta]$ become algebra homomorphisms.
\end{definition}

Let us also recall a useful formula describing multiplication in $R[z; \alpha, \delta]$. For any $k, n \in \ZZ_{>0}$ with $k \leq n$, let $S_{n,k}: R \to R$ denote an operator defined as the sum of all $\binom{n}{k}$ different compositions of $k$ differentiations $\delta$ and $n-k$ homomorpisms $\alpha$. Then for any $r\in R$ we have the following formula for how to commute $z^n$ and $r$ (see \cite[\S 4.1]{Pir08b}: \begin{equation}\label{commute}z^nr=\sum_{k=0}^{n}S_{n,k}(r)z^{n-k}.\end{equation}

Turning back to our examples, we see that the quantum plane from Example \ref{ex 1} is the Ore extension $\mathbb{C}[x][y; \alpha, 0]$, where $\alpha (x)=q^{-1}x$ and the commuting relation becomes $yx=q^{-1}xy$. The universal enveloping algebra $U(\mathfrak{g})$ from Example \ref{ex 2} is the Ore extension $\mathbb{C}[y][x; id, y \frac {d} {dy}]$, and the commuting relation becomes $xy=yx+y$.

\subsubsection{Analytic Ore extensions}Next we consider a locally convex counterpart of the algebra $R[z;\alpha, \delta]$ - an analytical Ore extension ${\mathcal {O}}(\mathbb{C}, R; \alpha, \delta)$ - and state the theorem telling us the conditions under which the algebra ${\mathcal {O}}(\mathbb{C}, R; \alpha, \delta)$ (or some variant of it) becomes the Arens-Michael envelope of $R[z;\alpha, \delta]$. Below we explain the key steps in the construction of ${\mathcal {O}}(\mathbb{C}, R; \alpha, \delta)$. This theoretical framework is explained in detail in \cite{Pir08b}.

First, we  recall the following two technical definitions.

\begin{definition}[{\cite[Definition 4.1]{Pir08b}}] Let $E$ be a vector space and let $\cc{T}$ be a family of linear operators on $E$. A seminorm on $E$ is $\cc{T}$-\emph{stable} if for any $T \in \cc{T}$ there exists $C>0$ such that $$\| Tv \|\leq C \|v\|$$ for every $v \in E$.
\end{definition}

\begin{definition}[{\cite[Definition 4.2]{Pir08b}}]\label{local} Let $E$ be a locally convex topological space. A family $\cc{T}$ of linear operators on $E$ is called \emph{localizable} if the topology on $E$ can be defined by a family of $\cc{T}$-stable seminorms. A single operator $T$ is called \emph{localizable} if the singleton family $\cc{T}=\{T\}$ is localizable. 
\end{definition}

Let now $R$ be a $\widehat {\otimes}$-algebra equipped with a localizable endomorphism $\alpha: R \to R$  and a localizable differentiation $\delta : R \to R$. The next two lemmas will show that we can equip the space $\cc{O}(\bb{C}, R)$ of $R$-valued entire functions with a ``twisted'' multiplication which coincides with the multiplication on the Ore extension $R[z; \alpha, \delta]$ when we restrict to the polynomial subspace in $\cc{O}(\bb{C},R).$ Recall that $\cc{O}(\bb{C},R)$ is isomorphic to the projective tensor product $R \widehat{\otimes}\cc{O}(\bb{C})$ both as a locally convex topological space and as a left $R$-$\widehat{\otimes}$-module. Explicitly, for any family of seminorms $\{\| \cdot \|_{\lambda} \}_{\lambda \in \Lambda}$ defining the topology on $R$, the space $\OO(\bb{C},R)$ is described as convergent Taylor series$$\{f(z)=\sum_n c_n z^n \quad : \quad c_n\in R, \quad \| f\|_{\lambda, \rho} < \infty \quad \text{for any} \quad \lambda \in \Lambda, \rho>0 \},$$where $\| f\|_{\lambda, \rho}=\sum_{n=0}^{\infty}\|c_n\|_{\lambda} \ \rho^n.$ The topology on $\cc{O}(\bb{C},R)$ is defined by the family of seminorms $$\{\| \cdot \|_{\lambda, \rho} \ : \ \lambda \in \Lambda, \ \rho \in \bb{R}_{>0} \}.$$

Consider the inclusion of locally convex topological vector spaces $$R[z; \alpha, \delta] \hookrightarrow \OO(\bb{C},R),$$where $R[z; \alpha, \delta]$ is equipped with the ``twisted multiplication'' from Definition \ref{twisted mult} and the induced topology from $\OO(\bb{C},R)$. The next lemma shows that we can extend \eqref{commute} from the dense subspace $R[z; \alpha, \delta]$ to the whole space $\OO(\bb{C},R)$.

\begin{lemma}[{\cite[Lemma 4.2]{Pir08b}}] Suppose $R$ is a $\widehat {\otimes}$-algebra, $\alpha: R \to R$ - a localizable endomorphism and $\delta : R \to R$ - a localizable differentiation.   Then there exists a unique continuous linear map 
$$\tau: \mathcal{O}(\mathbb{C}) \widehat{\otimes} R \to R \widehat{\otimes}\mathcal{O}(\mathbb{C}),$$ such that
$$\tau(z^n \otimes r)= \sum \limits_{k=0}^{n} S_{n,k}(r) \otimes z^{n-k} \quad \text{for all} \quad r \in R \quad \text{and} \quad  n \in \mathbb{Z}_{\geq 0}.$$ 
\end{lemma}
 
Now set $A=\OO(\bb{C},R) \cong R \widehat{\otimes} \mathcal{O} (\mathbb{C})$. Define the multiplication map $m_A:A \widehat{\otimes} A \to A$ as the composition
$$R \widehat{\otimes} \mathcal{O}(\mathbb{C}) \widehat{\otimes} R \widehat{\otimes} \mathcal{O}(\mathbb{C}) \xrightarrow {\bf {1}_{R} \otimes \tau \otimes \bf {1}_{ \mathcal{O}(\mathbb{C}) }} R \widehat{\otimes}R\widehat{\otimes} \mathcal{O}(\mathbb{C})\widehat{\otimes} \mathcal{O}(\mathbb{C}) \xrightarrow {m_R \widehat{\otimes} m_{\mathcal{O}(\mathbb{C})}}R \widehat{\otimes} \mathcal{O}(\mathbb{C}).$$

\begin{proposition}[{\cite[Proposition 4.3]{Pir08b}}] The map  $m_A:A \widehat{\otimes} A \to A$ turns $A=\OO(\bb{C},R)$ into a $\widehat {\otimes}$-algebra, such that the inclusion map $ i: R[z;\alpha, \delta] \hookrightarrow \OO(\bb{C},R) $ is an algebra homomorphism. 
\end{proposition}

The last proposition allows us to give the following definition.

\begin{definition}[{\cite[Definition 4.3]{Pir08b}}] The algebra $A=R \widehat{\otimes} \mathcal{O}(\mathbb{C})$ with the above multiplication map will be denoted ${\mathcal {O}}(\mathbb{C}, R; \alpha, \delta)$ and called  {\it the analytical Ore extension} of algebra $R$. 
\end{definition}
Note that ${\mathcal {O}}(\mathbb{C}, R; \alpha, \delta)$ contains $R$ as a closed subalgebra and is therefore a $R$-$\widehat {\otimes}$-algebra.

Next, we strengthen the above result in the case when $R$ is moreover an Arens-Michael algebra. First, we have the following refinement of Definition \ref{local}.

\begin{definition}[{\cite[Definition 4.4]{Pir08b}}] Let $R$ be an Arens-Michael algebra. A family $\cc{T}$ of linear operators on $R$ is called $m$-\emph{localizable} if the topology on $R$ can be defined by a family of $\cc{T}$-stable submultiplicative seminorms. A single operator $T$ is called $m$-\emph{localizable} if the singleton family $\cc{T}=\{T\}$ is $m$-localizable. 
\end{definition} 
The next proposition shows that if $R$ is  an Arens-Michael algebra and operators $\alpha$ and $\delta$ form an $m$-localizable family, then the analytic Ore extension ${\mathcal {O}}(\mathbb{C}, R; \alpha, \delta)$ is itself an Arens-Michael algebra.

\begin{proposition}[{\cite[Proposition 4.5]{Pir08b}}] Let $R$ be an Arens-Michael algebra, $\alpha : R \to R$ - an algebra endomorphism, $\delta : R \to R$ - an $\alpha$-differentiation. Suppose that the set $\{\alpha, \delta \}$ is $m$-localizable. Then ${\mathcal {O}}(\mathbb{C}, R; \alpha, \delta)$ is an Arens-Michael algebra. 
\end{proposition}

Now suppose $R$ is an algebra (without a topology), $\alpha=id : R \to R$ is the identity map and $\delta : R \to R$ is a differentiation. Denote by $R_\delta$ the Arens-Michael algebra, obtained as the completion of $R$ by the system of all $\delta$-stable submiltiplicative seminorms. Let $j$ be the canonical homomorphism $j: R \to R_\delta$. Clearly $\delta$ defines a unique $m$-localizable differentiation $\widehat {\delta}$ of $R_\delta$ with the property $ \widehat {\delta} \circ j=j \circ \delta$. Therefore we get homomorphisms $$R[z; id, \delta] \to R_\delta[z; id, \widehat {\delta}] \hookrightarrow {\mathcal {O}}(\mathbb{C}, R_\delta; id, \widehat {\delta}),$$  where the first one coincides with $j$ on $R$ and maps $z$ to $z$, and the second one is a canonical inclusion. Let $\iota_{R[z; id, \delta]}$ be the composition homomorphism. The next result describes the Arens-Michael envelope of $R[z; id, \delta]$ as the analytic Ore extension $\mathcal {O}(\mathbb{C}, R_\delta; id, \widehat {\delta})$.

\begin{theorem}[{\cite[Theorem 5.1]{Pir08b}}]\label{Pirk}The pair $(\mathcal {O}(\mathbb{C}, R_\delta; id, \widehat {\delta}), \iota_{R[z; id, \delta]})$ is the Arens-Michael envelope of  the algebra $R[z; id, \delta]$.
\end{theorem}

The situation when $\alpha \neq id$ is harder to handle. Let $R$ be an algebra, let $\alpha:R \to R$ be its endomorphism and let $X$ be an $R$-bimodule. We will denote by $_{\alpha}X$ the $R$-bimodule obtained by endowing the underlying abelian group of $X$ with a new $R$-multiplication rule $\bullet$: $$r \bullet x=\alpha(r)x, \quad x \bullet r=xr, \quad r\in R, \quad x\in X.$$

Let again $R$ be an algebra without a topology, $\alpha:R\to R$ be an endomorpism of $R$, and $\delta:R \to R$ be an $\alpha$-differentiation. By applying the Arens-Michael functor to $\alpha$ we obtain an endomorphism $$\widehat{\alpha}: \widehat{R} \to \widehat{R}$$of the Arens-Michael envelope of $R$ satisfying $$\widehat{\alpha} \circ \iota_R=\iota_R \circ \alpha.$$
Since we can view $\iota_R: R \to \widehat{R}$ as a morphism $_{\alpha}R \to _{\widehat{\alpha}}\widehat{R}$ of $R$-bimodules, we get that the composition $$R \xrightarrow{\delta}R  \xrightarrow{\iota_R}  _{\widehat{\alpha}}\widehat{R}$$ is a differentiation. By applying the universal property of the Arens-Michael envelopes (or rather its version for $R$-modules, see \cite[Definition 3.2]{Pir08b}), we obtain a unique $\widehat{\alpha}$-differentiation $$\widehat{\delta}: \widehat{R} \to   _{\widehat{\alpha}}\widehat{R}$$ satisfying $$\widehat{\delta} \circ \iota_R=\iota_R \circ \delta.$$

We have the following general result describing the Arens-Michael envelope of the algebraic Ore extension as an analytic Ore extension when certain technical conditions are met.
\begin{theorem}[{\cite[Theorem 5.17]{Pir08b}}]\label{main theorem} In the above setup, if the family  $\{\widehat{\alpha},\widehat{\delta}\}$ is $m$-localizable, then there exists a unique $R$-homomorphism
$$\iota_{R[z;\alpha,\delta]}: R[z; \alpha, \delta] \to {\mathcal {O}}(\mathbb{C}, \widehat{R}; \widehat{\alpha}, \widehat{\delta}) \quad \text{such that} \quad z \mapsto z.$$
The algebra ${\mathcal {O}}(\mathbb{C}, \widehat{R}; \widehat{\alpha}, \widehat{\delta})$ with the homomorphism $\iota_{R[z;\alpha,\delta]}$ is the Arens-Michael envelope of $R[z; \alpha, \delta]$.

\end{theorem}

\section{The Arens-Michael envelope of the Jordanian plane}\label{JP}

Now we turn to the main results of this paper. We first compute the Arens-Michael envelope of the Jordanian plane.

\begin{definition}\label{Jordanian plane} \emph{The Jordanian plane} over $\mathbb{K}$ is the $\mathbb{K}$-algebra $\Lambda_2 (\mathbb{K})$ given by generators $x$ and $y$ and a commuting relation $yx=xy+y^2$, i.e.  $$\Lambda_2 (\mathbb{K})=\mathbb{K} \langle x,y \rangle / (yx-xy-y^2) .$$

\end{definition}

We would be interested in the case when $\mathbb{K}=\mathbb{C}$ and we would denote $\Lambda_2 (\mathbb{C}) =:\Lambda_2 $.

Following a simple induction argument it is easy to check that the monomials $\{x^i y^j \ | \ i,j \in \mathbb{Z}_+\}$  span  $\Lambda_2 $. It is shown in \cite{Shi05} that they are also linearly independent and, therefore, form the basis of $\Lambda_2 $. As a result, we can again view $\Lambda_2 $ as a polynomial algebra with a ``twisted'' multiplication.

Comparing Definition \ref{Jordanian plane} to Definition \ref{twisted mult} we see that the  Jordanian plane is the Ore extension $\mathbb{C}[y][x; id, -y^2 \frac {d} {dy}]$. Therefore, in order to apply Theorem \ref{Pirk} we need  to describe  the system of all submiltiplicative $\delta$-stable seminorms on $\mathbb{C}[y]$, where $\delta= -y^2 \frac {d} {dy}$. 

Write an element $a \in \mathbb{C}[y]$ as a polynomial $\sum \limits_{i=0}^{n}  a_i y^i$. We first have the following result.

\begin{proposition} The family $$\{\| \cdot \|_{\rho}  \quad : \quad  \|a\|_\rho = \sum \limits_{i=0}^{n} | a_i| \frac{1}{(i-1)!} \rho^i < \infty,   \quad \rho \in \bb{R}_{>0} \}, \quad (-1! :=1, 0! :=1)$$ is equivalent to the family of all submultiplicative  $\delta$-stable seminorms on $\mathbb{C}[y]$ with $\delta=-y^2 \frac {d} {dy}$.
\end{proposition}

\begin{proof} We split the proof of the proposition into steps for the reader's convenience. 

\noindent{\bf Step 1.} First note that $\| \cdot \|_\rho $ is indeed a seminorm. Moreover, it is submultiplicative:

$$\| y^k y^l\|_\rho=\|y^{k+l}\|_\rho=\frac{\rho^{k+l}}{(k+l-1)!} \leq \frac{\rho^k}{(k-1)!} \frac{\rho^l}{(l-1)!}=\|y^k\|_\rho\|y^l\|_\rho, \quad \forall k,l \geq 0.$$ 
Now for $a,b \in \mathbb{C}[y]$ we have: 
$$
\begin{aligned}
\| ab \|_{\rho} &= \| \sum_{i=0}^{n}  a_i y^i \cdot \sum_{j=0}^{m}  b_j y^j\|_\rho= \| \sum \limits_{i,j=0}^{n,m}  a_i b_j y^i y^j \|_\rho  \\
 &\leq \sum \limits_{i,j=0}^{n,m}  |a_i| |b_j|\| y^i y^j\|_\rho \leq \sum \limits_{i,j=0}^{n,m}  |a_i| |b_j|\| y^i\|_\rho \| y^j\|_\rho \\
 &= \sum \limits_{i=0}^{n}  |a_i|\| y^i\|_\rho \cdot \sum \limits_{j=0}^{m}  |b_j| \|y^j\|_\rho =\| \sum \limits_{i=0}^{n}  a_i y^i\|_\rho \cdot \|\sum \limits_{j=0}^{m}  b_j y^j\|_\rho=\|a\|_\rho \|b\|_\rho.
\end{aligned}
$$

\noindent{\bf Step 2.} Next, we show that $\| \cdot \|_\rho $ is $\delta$-stable. Note that $$\delta (y^i)=-y^2 \frac {d} {dy} (y^i)=-iy^{i+1}, \quad \forall i \geq 1$$ and $$\delta (y^0)=\delta (1)=-y^2 \frac {d} {dy} (1)=0.$$  We have 
$$
\begin{aligned} \|\delta(a)\|_\rho&=\|\delta(\sum \limits_{i=0}^{n}  a_i y^i)\|_\rho=\|\sum \limits_{i=0}^{n}  a_i \delta(y^i)\|_\rho=\|\sum \limits_{i=1}^{n}  a_i (-iy^{i+1})\|_\rho \\
  &=\|\sum \limits_{i=1}^{n}  a_i iy^{i+1}\|_\rho=\|\sum \limits_{j=2}^{n+1}  a_{j-1} (j-1)y^{j}\|_\rho=\sum \limits_{j=2}^{n+1}  |a_{j-1}| \frac{j-1}{(j-1)!} \rho^{j} \\
 &=\rho \sum \limits_{j=2}^{n+1}  |a_{j-1}| \frac{1}{(j-2)!} \rho^{j-1} \leq \rho \sum \limits_{i=1}^{n}  |a_{i}| \frac{1}{(i-1)!} \rho^{i} \leq \rho \sum \limits_{i=0}^{n}  |a_{i}| \frac{1}{(i-1)!} \rho^{i}= \rho \|a\|_\rho.
 \end{aligned}
$$

\noindent{\bf Step 3.} Finally, we show that any submultiplicative $\delta$-stable seminorm $\| \cdot \|$ is dominated by $\| \cdot \|_\rho$ for some $\rho > 0$. 

Note that by induction we get $$\delta^j (y)=(-1)^j \ j! \ y^{j+1}, \quad j \geq 1.$$      
From $\delta$-stability we get $$\|\delta^j (a)\| \leq C \|\delta^{j-1} (a)\| \leq ... \leq C^j \|a\|.$$
Now setting $a =y$ we have 
$$\|\delta^j(y)\|=j!\|y^{j+1}\| \leq C^j \|y\|,$$ so that $$\|y^{j+1}\| \leq \frac{C^j \|y\|}{j!}, \quad j \geq 1.$$
Note that if we pick $C \geq 1$ and set $\rho :=C \max\{\|y\|, 1\} $, we have 
$$C^j \|y\| \leq C^j \max\{\|y\|, 1\} \leq C^j (\max\{\|y\|, 1\})^j  = \rho^j,$$ and therefore 
$$\|y^{j+1}\| \leq \frac{C^j \|y\|}{j!} \leq \frac{\rho^j}{j!} \leq \frac{\rho^{j+1}}{j!} = \|y^{j+1}\|_\rho, \quad j \geq 1.$$

For $j=0$ we have $$\|y\|=\frac{\|y\|}{\rho}\rho\leq D\|y\|_\rho,$$ where $D=\max\{\frac{\|y\|}{\rho}, 1 \} $. 

Finally, 
$$\|a\|=\|\sum \limits_{i=0}^{n}  a_i y^i\| \leq \sum \limits_{i=0}^{n}  |a_i| \| y^i\| \leq |a_0|+|a_1|D\|y\|_\rho + \sum \limits_{i=2}^{n}  |a_i| \| y^i\|_\rho \leq D \sum \limits_{i=0}^{n}  |a_i| \| y^i\|_\rho=D\|a\|_\rho.$$ \end{proof}
Next, we pass to a simpler family of seminorms.
\begin{lemma}\label{equiv} The family of seminorms on $\bb{C}[y]$ $$P= \{ \| \cdot \|_\rho  \quad : \quad   \|a\|_\rho = \sum \limits_{i=0}^{n} | a_i| \frac{1}{(i-1)!} \rho^i < \infty, \quad \rho \in \bb{R}_{>0},\}, \quad (-1! :=1, 0! :=1)$$ is equivalent to the family $$Q= \{ \| \cdot \|_q \quad :  \quad \|a\|_q = \sum \limits_{i=0}^{n} | a_i| \frac{1}{i!} q^i < \infty, \quad q \in \bb{R}_{>0}  \}, \quad (0! :=1),$$where $a=\sum \limits_{i=0}^{n}  a_i y^i \in \bb{C}[y].$ 
\end{lemma}

\begin{proof}
First, we observe $$\sum \limits_{i=0}^{n} | a_i| \frac{1}{i!} q^i \leq \sum \limits_{i=0}^{n} | a_i| \frac{1}{(i-1)!} q^i,$$ and, therefore, $Q \prec P$. 

Since for $i \geq 0$ $$ i \leq 2^i \Leftrightarrow \frac{1}{(i-1)!} \leq \frac {2^i}{i!},$$ we have $$\sum \limits_{i=0}^{n} | a_i| \frac{1}{(i-1)!} \rho^i \leq \sum \limits_{i=0}^{n} | a_i| \frac{2^i}{i!} \rho^i = \sum \limits_{i=0}^{n} | a_i| \frac{1}{i!} q^i, \quad q=2\rho,$$ and $P \prec Q$. 
\end{proof}

Finally, we can apply Theorem \ref{Pirk} to get the description of the Arens-Michael envelope of the Jordanian plane $\Lambda_2$.

\begin{theorem} The Arens-Michael envelope of the Jordanian plane $\Lambda_2$ is $$\widehat{\Lambda}_2:=\{a=\sum \limits_{i,j=0}^{\infty} a_{ij} x^i y^j \quad : \quad  \|a\|_\rho< \infty   \quad \text{for any} \quad \rho>0\},$$
where $$\|a\|_\rho=  \sum \limits_{i,j=0}^{\infty} | a_{ij}| \frac{1}{j!} \rho^{i+j}.$$

The topology on  $\widehat{\Lambda}_2$ is generated by the system $\{\| \cdot\|_\rho \ :  \ \rho \in \bb{R}_{>0}\},$ and multiplication is characterized by the relation $yx=xy+y^2$.
\end{theorem}

\begin{proof} By Theorem \ref{Pirk}, the Arens-Michael envelope of $\Lambda_2=\mathbb{C}[y][x; id, -y^2 \frac {d} {dy}]$ is given by the analytic Ore extension $\mathcal {O}(\mathbb{C}, \bb{C}[y]_\delta; id, \widehat {\delta})$, where $\bb{C}[y]_{\delta}$ is the completion of $\bb{C}[y]$ with respect to the family $Q$ of seminorms by Lemma \ref{equiv}. 

By the discussion preceding Theorem \ref{Pirk}, the space $\mathcal {O}(\mathbb{C}, \bb{C}[y]_\delta; id, \widehat {\delta})$ can be described as the set$$\{ a=\sum_{i=0}^{\infty} \sum_{j=0}^{\infty} a_{ij} x^i y^j \quad : \quad \| a \|_{q,\rho}< \infty \quad \text{for any} \quad q, \rho>0\},$$where $$\quad \|a \|_{q, \rho}=\sum_{i=0}^{\infty}\|\sum_{j=0}^{\infty} a_{ij}  y^j\|_{q} \ \rho^i=\sum_{i=0}^{\infty}\sum_{j=0}^{\infty} |a_{ij}|\frac{1}{j!}  q^j\ \rho^i,$$ with topology given by 
the family of seminorms $\{ \| \cdot \|_{q,\rho} \ : \ q,\rho\in \bb{R}_{>0} \}.$ It is easy to check that this description is equivalent to the description in the statement of the theorem. 
\end{proof}
We see that again the Arens-Michael envelope of a polynomial algebra with a ``twisted'' multiplication is a power series algebra with the same multiplication rule.

\section{The Arens-Michael envelope of $U_q(\mathfrak{sl}(2)), \quad |q|=1$.}\label{EA}

In the section, we turn to the second main result of this paper. 

The quantum enveloping algebra $U_q(\mathfrak{sl}(2))$ of the Lie algebra $\mathfrak{sl}(2)$ plays one of the most important roles in  noncommutative geometry. It can be defined in two slightly different ways. 

\begin{definition}\label{first def} For $q \in \bb{C} \setminus \{1,-1 \}$, consider an algebra $U_q(\mathfrak{sl}(2))$ on four generators $E, F, K, K^{-1}$ subject to the following relations:

\begin{enumerate}

\item $K K^{-1}=K^{-1} K=1$,

\item $KEK^{-1}=q^2 E, \ KFK^{-1}=q^{-2}F$,

\item $[E,F]= \frac {K-K^{-1}}{q-q^{-1}}$.
\end{enumerate}

\end{definition}

\begin{definition}\label{second def} For $q \in \bb{C} \setminus \{1,-1 \}$, consider an algebra $U'_q(\mathfrak{sl}(2))$ on five generators $E,F,K,K^{-1},L$ with relations:

\begin{enumerate}

\item $K K^{-1}=K^{-1} K=1$,

\item $KEK^{-1}=q^2 E, \ KFK^{-1}=q^{-2}F$,

\item $[E,F]= L, \ (q-q^{-1})L=K-K^{-1}$,

\item $[L,E]=q(EK+K^{-1}E) , \ [L,F]=-q^{-1}(FK+K^{-1}F)$.
\end{enumerate}

\end{definition}

Clearly, $U'_q(\mathfrak{sl}(2))$ is isomorphic to $U_q(\mathfrak{sl}(2))$ via a map that sends $L \in U'_q(\mathfrak{sl}(2))$ to $[E,F] \in U_q(\mathfrak{sl}(2))$ and leaves other generators intact.
Note that the second definition allows us to consider the limiting case $q=1$ where our quantum enveloping algebra almost becomes the usual enveloping algebra $U(\mathfrak{sl}(2))$. In fact, we have the following isomorphisms (see \cite{Kas94}):
\begin{equation}\label{envel}
U'_1(\mathfrak{sl}(2)) \cong \frac{U(\mathfrak{sl}(2)) [K]}{(K^2-1)} \cong U(\mathfrak{sl}(2)) \otimes \frac {\mathbb{C} [K]}{(K^2-1)}.\end{equation}  For more information on $U_q(\mathfrak{sl}(2))$, see \cite{Kas94}.

We obtain the following description of the Arens-Michael envelope of $U_q(\mathfrak{sl}(2))$ for $|q|=1$.

\begin{theorem} Let $U_q(\mathfrak{sl}(2))$ be the quantum enveloping algebra of $\mathfrak{sl}(2)$, and let $\widehat{U_q(\mathfrak{sl}(2))}$ be its Arens-Michael envelope.

\begin{enumerate}
\item If $q=1$, then $$\widehat{U_1(\mathfrak{sl}(2)} \cong \widehat{U(\mathfrak{sl}(2)} \widehat{\otimes} \frac {\mathbb{C} [K]}{(K^2-1)}.$$

\item
$ |q|=1, q \neq 1,-1$, then 
$$
\widehat{U_q(\mathfrak{sl}(2))} =
\{ c = \sum \limits_{i \in \mathbb{Z}, n,m \geq 0} c_{i,n,m} K^i F^n E^m \quad : \quad \|c\|_\rho < \infty \quad \text{for any} \quad \rho > 0 \},
$$
where 
$$\|c\|_\rho=  \sum \limits_{i \in \mathbb{Z}, n,m \geq 0} |c_{i,n,m}| \rho^{i+n+m}$$
The topology on  $\widehat{U_q(\mathfrak{sl}(2))}$  is generated by the system $\{\| \cdot\|_\rho \ :  \ \rho \in \bb{R}_{>0}\}.$
\end{enumerate}
\end{theorem}
\begin{proof}We split the proof of the theorem into two cases according to how the theorem is stated.

{\it Proof of (1).}  
Definition \ref{second def} allows us to consider the limiting case $q=1$ where our quantum enveloping algebra almost becomes the usual enveloping algebra $U(\mathfrak{sl}(2))$ as we mentioned in \eqref{envel}.

The representation of $U'_1(\mathfrak{sl}(2))$ as the tensor product of two algebras for which the Arens-Michael envelope is already known allows us to quickly compute the Arens-Michael envelope of $U'_1(\mathfrak{sl}(2))$ using the results of \S \ref{tensor}. Specifically, endow $U'_1(\mathfrak{sl}(2))$ with the strongest locally convex topology $\tau_{str}$. By Proposition \ref{tensor 1}, we have the following isomorphism of $\widehat{\otimes}$-algebras $$(U'_1(\mathfrak{sl}(2)), \tau_{str})\cong \left(U(\mathfrak{sl}(2)), \tau_{str}\right) \widehat{\otimes}\left( \frac {\mathbb{C} [K]}{K^2-1},\tau_{str}\right).$$
Next, we apply Proposition \ref{tensor 2} (along with Definition \ref{AM2}) to the above projective tensor product to get  $$\widehat{U'_1(\mathfrak{sl}(2))}\cong \widehat{U(\mathfrak{sl}(2))} \widehat{\otimes}\left( \frac {\mathbb{C} [K]}{K^2-1}\right)^{\widehat{  \  \ }}.$$

Since $\frac{\mathbb{C} [K]}{(K^2-1)}$ is a finite-dimensional vector space, its  Arens-Michael envelope coincides with  $\frac{\mathbb{C} [K]}{(K^2-1)}$. Therefore, we finally get
$$\widehat{U'_1(\mathfrak{sl}(2)} \cong \widehat{U(\mathfrak{sl}(2)} \widehat{\otimes} \frac {\mathbb{C} [K]}{(K^2-1)}.$$
The Arens-Michael envelope $\widehat{ U(\mathfrak{sl}(2)}$ was shown  to be the direct product of matrix algebras in Example \ref{lie matrix}.

{\it Proof of (2).}
For $q \neq 1,-1$, note that the algebra $U_q(\mathfrak{sl}(2))$ is an iterated Ore extension (see \S \ref{Ore}). Namely, start with 
$$A_0= \mathbb{C} [K,K^{-1}]$$along with a $\bb{C}$-linear algebra homomorphism $\alpha_0:A_0 \to A_0$ defined by $$ \alpha_0(K)=q^2K.$$
Next, consider the Ore extension
$$A_1=A_0[F,\alpha_0,\delta_0=0]$$ equipped with a  $\bb{C}$-linear algebra homomorphism $\alpha_1$ and a $\bb{C}$-linear differentiation $\delta$ defined by $$ \alpha_1(F^jK^l)=q^{-2l}F^jK^l, \quad  \delta (K)=0, \quad \delta(F^jK^l)
=\sum_{i=0}^{j-1} F^{j-1}  \ \delta_{q^{-2i}K}(F) \ K^l,$$where $ \delta_K(F)= \frac{K-K^1}{q-q^-1}$ is a Laurent polynomial in $K$. Finally, consider the Ore extension$$A_2=A_1[E,\alpha_1,\delta].$$ One easily checks by comparing the above sequence of Ore extensions to Definition \ref{first def} that $$A_2 \cong U_q(\mathfrak{sl}(2)).$$
Therefore, we can apply the results of \S \ref{Ore} to each consecutive Ore extension to calculate the Arens-Michael envelope of $A_2$. 

The Arens-Michael envelope of the algebra $A_0= \mathbb{C} [K,K^{-1}]$ is very simple:  
$$\widehat{A_0}=\{a=\sum \limits_{i \in \mathbb{Z}} a_{i} K^i \quad : \quad \|a\|_\rho=  \sum \limits_{i \in \mathbb{Z}} | a_{i}|  \rho^{i} < \infty   \quad \text{for all} \quad \rho>0\},$$
where $\|a\|_\rho=  \sum \limits_{i \in \mathbb{Z}} | a_{i}|  \rho^{i}$, and the topology is generated by the family $\{ \| \cdot \|_\rho \ : \ \rho \in \bb{R}_{>0} \}$.

After extending $\alpha_0: A_0 \to A_0$  to $\widehat{\alpha_0}:\widehat{A_0} \to \widehat{A_0}$, we check that $\widehat{\alpha_0}$ is indeed $m$-localizable (use $|q|=1$):
$$\|\widehat{\alpha_0} (\sum \limits_{i \in \mathbb{Z}} a_{i} K^i) \|_\rho = \| \sum \limits_{i \in \mathbb{Z}} a_{i} q^{2i}K^i\|_\rho=\sum \limits_{i \in \mathbb{Z}} |a_{i}||q^{2i}| \rho^i =\sum \limits_{i \in \mathbb{Z}} |a_{i}| \rho^i =\|\sum \limits_{i \in \mathbb{Z}} a_{i} K^i \|_\rho.$$
Applying Theorem \ref{main theorem}, we obtain 
$$
\begin{aligned}
\widehat{A_1} &= {\mathcal {O}}(\mathbb{C}, \widehat{A_0}; \widehat{\alpha_0}, 0) \\ &= \{b=\sum \limits_{i \in \mathbb{Z}, n \geq 0} b_{i,n} K^i F^n \quad : \quad \|b\|_\rho < \infty \quad  \text{for any} \quad \rho > 0 \},
\end{aligned}
$$
where $\|b\|_\rho=   \sum \limits_{i \in \mathbb{Z}, n \geq 0} |b_{i,n}| \rho^{i+n}$, and the topology is generated by the family $\{ \| \cdot \|_\rho \ : \ \rho \in \bb{R}_{>0} \}$. 

Finally, the operators $\alpha_1$ and $\delta$ simply extend to $\widehat{A_1}$ by their action on the generators $K$ and $F$. We check that $\{\widehat{\alpha_1}, \widehat{\delta}\}$ is an $m$-localizable family. In the calculations below, we make extensive use of the relations between the generators  $K^i F^n=q^{-2in}F^n K^l$.

$\widehat{\alpha_1}$ is $m$-localizable:
$$
\begin{aligned}
&\| \widehat{\alpha_1} (\sum \limits_{i \in \mathbb{Z}, n \geq 0} b_{i,n} K^i F^n) \|_\rho = \| \widehat{\alpha_1} (\sum \limits_{i \in \mathbb{Z}, n \geq 0} b_{i,n} q^{-2in} F^n K^i) \|_\rho \\ 
&= \| \sum \limits_{i \in \mathbb{Z}, n \geq 0} b_{i,n} q^{-2in} q^{-2i}F^n K^i \|_\rho = \sum \limits_{i \in \mathbb{Z}, n \geq 0} |b_{i,n}| |q^{-2in}| |q^{-2i}| \rho^{i+n} \\ 
&= \sum \limits_{i \in \mathbb{Z}, n \geq 0} |b_{i,n}| |q^{-2in}|  \rho^{i+n}=\| \sum \limits_{i \in \mathbb{Z}, n \geq 0} b_{i,n} q^{-2in} F^n K^i \|_\rho=\| \sum \limits_{i \in \mathbb{Z}, n \geq 0} b_{i,n}   K^i F^n \|_\rho .
\end{aligned}
$$ 

Before proving that $\widehat{\delta}$ is $m$-localizable, let us perform one auxiliary computation:
$$
\begin{aligned}
&\delta (K^i F^n) =\delta(q^{-2in}F^n K^i)=q^{-2in}\sum_{j=0}^{n-1} F^{n-1}  \ \delta_{q^{-2j}K}(F) \ K^i= \\ 
&= q^{-2in}F^{n-1}K^i \ \frac{1}{q-q^{-1}} \ \left[(K-K^{-1})+(q^{-2}K-q^2 K)+ \ ... \ +(q^{-2(n-1)}K-q^{2(n-1)}K)\right]  \\ 
&=q^{-2in}F^{n-1}K^i \ \frac{1}{q-q^{-1}} \ \left[K(1+q^{-2}+ \ ... \ +q^{-2(n-1)})-K^{-1}(1+q^2+ \ ... \ + q^{2(n-1)} \right]  \\ 
&=q^{-2in}F^{n-1}K^i \ \frac{1}{q-q^{-1}} \ \left[K \left(\frac{1-q^{-2n}}{1-q^{-2}}\right)-K^{-1}\left(\frac{1-q^{2n}}{1-q^2}\right)\right].
\end{aligned}
$$
Now we can show that $\widehat{\delta}$ is $m$-localizable:
\footnotesize
$$
\begin{aligned}
&\| \widehat{\delta} (\sum \limits_{i \in \mathbb{Z}, n \geq 0} b_{i,n} K^i F^n) \|_\rho  = \|  \sum \limits_{i \in \mathbb{Z}, n \geq 0} b_{i,n} \  \delta (K^i F^n )  \|_\rho \\ 
&= \left\|  \sum \limits_{i \in \mathbb{Z}, n \geq 1} b_{i,n}   q^{-2in}F^{n-1}K^i  \frac{1}{q-q^{-1}} \left[K(\frac{1-q^{-2n}}{1-q^{-2}})-K^{-1}(\frac{1-q^{2n}}{1-q^2})\right]  \right\|_\rho  \\
&= \left\|  \sum \limits_{i \in \mathbb{Z}, n \geq 1} b_{i,n}   q^{-2in}\left[F^{n-1}K^{i+1}  \frac{1-q^{-2n}}{(q-q^{-1})(1-q^{-2})}  - F^{n-1}K^{i-1}\frac{1-q^{2n}}{(q-q^{-1})(1-q^2)}\right]  \right\|_\rho  \\
&= \left\|  \sum \limits_{i \in \mathbb{Z}, n \geq 0} \left[b_{i-1,n+1}   q^{-2(i-1)(n+1)}  \frac{1-q^{-2(n+1)}}{(q-q^{-1})(1-q^{-2})}  - b_{i+1,n+1}  q^{-2(i+1)(n+1)} \frac{1-q^{2(n+1)}}{(q-q^{-1})(1-q^2)}\right] F^n K^i \right\|_\rho  \\
&= \sum \limits_{i \in \mathbb{Z}, n \geq 0} \left\rvert  \left[b_{i-1,n+1} \  q^{-2(i-1)(n+1)}  \frac{1-q^{-2(n+1)}}{(q-q^{-1})(1-q^{-2})}  - b_{i+1,n+1} \ q^{-2(i+1)(n+1)} \frac{1-q^{2(n+1)}}{(q-q^{-1})(1-q^2)}\right]  \right\rvert  \rho^{i+n} \\
&\leq \sum \limits_{i \in \mathbb{Z}, n \geq 0}   \left[  \left\rvert b_{i-1,n+1} \  q^{-2(i-1)(n+1)}  \frac{1-q^{-2(n+1)}}{(q-q^{-1})(1-q^{-2})}  \right\rvert  + \left\rvert  b_{i+1,n+1} \ q^{-2(i+1)(n+1)} \frac{1-q^{2(n+1)}}{(q-q^{-1})(1-q^2)} \right\rvert \right] \rho^{i+n}  \\
&\leq\sum \limits_{i \in \mathbb{Z}, n \geq 0}   \left[  |b_{i-1,n+1}|  \left| q^{-2(i-1)(n+1)} \right| \frac{2}{|(q-q^{-1})(1-q^{-2})|}   +  | \ b_{i+1,n+1}|  \left|q^{-2(i+1)(n+1)} \right| \frac{2}{|(q-q^{-1})(1-q^2)|} \right]  \rho^{i+n} \\
&\leq \frac{2}{|(q-q^{-1})(1-q^{-2})|} \sum \limits_{i \in \mathbb{Z}, n \geq 0}   |b_{i,n}| \rho^{i+n} + \frac{2}{|(q-q^{-1})(1-q^{-2})|} \sum \limits_{i \in \mathbb{Z}, n \geq 0}   |b_{i+1,n+1}| \frac{\rho^{i+n+2}}{\rho^2}  \\
&\leq C  \sum \limits_{i \in \mathbb{Z}, n \geq 0}   |b_{i,n}| \rho^{i+n} + \frac{C}{\rho^2} \sum \limits_{i \in \mathbb{Z}, n \geq 0}   |b_{i,n}| \rho^{i+n} \leq (C + \frac{C}{\rho^2}) \left\| \sum \limits_{i \in \mathbb{Z}, n \geq 0} b_{i,n} K^i F^n \right\|_\rho .
\end{aligned}
$$

\normalsize
Applying Theorem \ref{main theorem}, we conclude that the Arens-Michael envelope of $U_q(\mathfrak{sl}(2))$ for ${|q|=1, \quad q \neq 1, -1}$ is 
$$
\begin{aligned}
\widehat{U_q(\mathfrak{sl}(2))} &= \widehat{A_2}={\mathcal {O}}(\mathbb{C}, \widehat{A_1}; \widehat{\alpha_1}, \widehat{\delta}) \\ &=
\{ c = \sum \limits_{i \in \mathbb{Z}, n,m \geq 0} c_{i,n,m} K^i F^n E^m \quad : \quad \|c\|_\rho< \infty \quad  \text{for any} \quad \rho > 0 \},
\end{aligned}
$$
where $\|c\|_\rho=  \sum \limits_{i \in \mathbb{Z}, n,m \geq 0} |c_{i,n,m}| \rho^{i+n+m}$, and the topology is generated by the family $$\{ \| \cdot \|_\rho \ : \ \rho \in \bb{R}_{>0} \}.$$ 
\end{proof}

\begin{remark}The case $|q| \neq 1$ was completely open at the time of the publication of \cite{Ped15} in 2015, and, to the best knowledge of the author, remains an open question in 2020.  However, see \cite[\S 4]{Kos17} for a recent discussion of the possible approaches for this case.\end{remark}



\end{document}